%% file: xmod_gwa.tex
\documentclass[a4paper,11pt]{article}
\usepackage[bookmarks=true,
              bookmarksnumbered=true, breaklinks=true,
              pdfstartview=FitH, hyperfigures=false,
              plainpages=false, naturalnames=true,
              colorlinks=true,
              pdfpagelabels]{hyperref}
\usepackage{geometry,latexsym,amssymb,amsmath,amsthm,color,bm}
\usepackage[latin5]{inputenc}
\usepackage{amssymb}%added
\usepackage{enumerate}
\usepackage{enumitem}
\usepackage[T1]{fontenc}
\usepackage{authblk}
\usepackage[all]{xy}
\usepackage{palatino}
\usepackage{indentfirst}
\usepackage{titlesec}
\usepackage{graphics}
\usepackage{tabularx}
\usepackage{lipsum}
\usepackage{longtable}
\usepackage{array}
\usepackage{fancyvrb,shortvrb}
\usepackage[ruled,vlined]{algorithm2e}

\input{newcomm}

\input{tcilatex.tex}

\usepackage{pifont}% http://ctan.org/pkg/pifont
\newcommand{\cmark}{\ding{51}}%
\newcommand{\xmark}{\ding{55}}%

\theoremstyle{plain}

\theoremstyle{definition}

\begin{document}
\title{2-Dimensional Groups with Action : The Category of Crossed Module of Groups
with Action}

\author[a]{A. Odaba\c{s}}
\author[a]{E. Soylu Yılmaz}
\affil[a]{\small{Department of Mathematics and Computer Science, Osmangazi University, Eskisehir, Turkey}}

\date{}

\maketitle

\begin{abstract}
	
 In this paper, we define the notion of crossed modules of groups with action and investigate related structures. Functions for computing of these structures have been written using the \verb|GAP| computational discrete algebra programming language.
	
\end{abstract}

\noindent{\bf Key Words:} Group with action; category; Gap computational discrete algebra programming language.
\\ {\bf Classification:} 18-08, 18G50.

\input{scw_hazirlik.tex}

\end{document}

%% file: newcomm.tex
\newcommand{\GAP}     {{\sf GAP}}

\newcommand{\catSimpGrp}    {{\bf SimpGrp}}

\newcommand{\catXMod}   {{\bf XMod}}
\newcommand{\catLie}   {{\bf Lie}}
\newcommand{\catLeibniz}   {{\bf Leibniz}}
\newcommand{\catGrLie}   {{\bf Gr_{Lie}}}
\newcommand{\catGrwA}   {{\bf Gr^{\bullet }}}
\newcommand{\catXModwA}   {{\bf XMod^{\bullet }}}
\newcommand{\catSimpGrpwA}   {{\bf SimpGrp^{\bullet }}}

\newcommand{\calX}{\mathcal{X}}

%% file: scw_hazirlik.tex
%TCIDATA{Version=5.50.0.2890}
%TCIDATA{LaTeXparent=0,0,bos_preamble.tex}

\section{Introduction}

Crossed modules can be considered as two-dimensional of algebraic structures
were first mentioned by Whitehead in 1941 \cite{WHT1}. Later Whitehead named
crossed modules in \cite{WHT2} as an additional note of his early work.
These concepts aroused in the paper entitled 'Combinatorial homotopy II' 
\cite{WHT} paper which also introduced the substantial concept of a free
crossed module. Many generalizations of crossed module are given in the
paper of Janelidze \cite{JL}.

The term 2-group (two-dimensional group) can be considered as a cluster of
well-matched categories of crossed modules and also the cat-1 groups which
are the pair of groups. Both crossed modules and cat-1 groups can be viewed
as a Moore complex of a simplicial group.

The notion of the group with action first appeared in Datuashvili's paper 
\cite{DT}. Datuashvili demonstrated this structure by one of the three
problems of Loday mentioned in \cite{LO,LOD}. The problem is associated with
the Leibniz algebras to define the algebraic correspondence called
`coquecigrue' object as Lie group structure. Leibniz algebras are closely
related with the lower central series of a group constructed and examined in 
\cite{WITT} by Witt.

Leibniz algebras was firstly identified with the research of Loday in 1989 
\cite{LD}. A Leibniz algebra is a non-associative equivalent of a Lie
algebra. To explain this association, the functor 
\[
F: \catGrLie \longrightarrow \catLie 
\]%
was constructed by Witt \cite{WITT} in 1937 where $\catGrLie$ is the
category of Lie groups and $\catLie$ is the category of Lie algebras. A
group with action arises from the analogous version of the above functor
taking the Leibniz algebras instead of Lie algebras. Since Leibniz algebras
are non-associative congruence of Lie algebras, a group with action takes
place of Lie groups. Therefore, Datuashvili defined the functor 
\[
\catGrwA \longrightarrow \catLeibniz 
\]%
from the category, $\catGrwA$, of group with action to the category, $%
\catLeibniz$, of Leibniz algebras (see for details \cite{DT,DT2}).

Likewise obtaining a group version of Lie algebras, the category of
Leibniz algebras are equivalent to the category of group with action.

A shared package \textsf{XMod} \cite{xmod}, for the \verb|GAP| \cite{gap}, computational discrete algebra system was described by Wensley et al. which contains functions for computing crossed modules of groups and cat$^{1}$-groups and their morphisms. Thereafter, the algebraic version of a \verb|GAP| package XModAlg \cite{xmodalg} was given by Arvasi and Odabas (see \cite{arvasi-odabas}). In
this paper, we describe a package \textsf{XModGwA} for \verb|GAP| which constructs crossed modules of groups with action (see \cite{xmodgwa}).

In this paper, we investigate the simplicial group with action corresponding
to a crossed module of group with action. We also give a natural equivalence
for these structures.

Shortly, we can summarize the purpose of this paper as:
\begin{itemize}
	\item To construct the two - dimensional group with action as crossed
	module of group with action,
	
	\item To determine the action conditions between two group with
	actions,
	
	\item To compose a group with action crossed module by an action,
	
	\item To give the categorical equivalences between other structures.
\end{itemize}

\section{The Main Text}

The concept of crossed module, generalizing the notion of a G-module, was
introduced by Whitehead \cite{WHT} in the course of his studies on the
algebraic structure on the second relative homotopy group. We first recall
the crossed modules of groups from \cite{TIM,arvasi3}. A crossed module of
groups is a homomorphism $\partial :S\rightarrow R$ where $S$ and $R$ are
groups and there exist an action of $R$ on $S$ denoted by $(r,s)\mapsto $ $%
r\cdot s$. These data must satisfy the following two conditions,

\textbf{CM1)} $\partial $ is $R$-equivariant, so $\partial (r\cdot s)=r{%
	\partial (}s)r^{-1}$

\textbf{CM2)} Peiffer rule; ${\partial (}s)\cdot s_{1}=ss_{1}s^{-1}$

\noindent where $r\in R$ and $s,s_{1}\in S$.

A crossed module is written by $\calX=(\partial :S\rightarrow R)$ notation.
The groups $S,R$ and the group homomorphism $\partial $ are called the \emph{%
	source}, \emph{range} and \emph{boundary} of $\calX$ respectively. When only
the first of these conditions is satisfied, the resulting structure is a 
\emph{pre-crossed module}. Given a pre-crossed module $\calX=(\partial
:S\rightarrow R)$, one can form an internal directed graph in the category
of groups simply by forming the semidirect product $S \rtimes R$ and taking
the source and target to send an element $(s,r)$ to $r$ and $\partial (s)r$
respectively.

\example{}

\begin{itemize}
	
	\item Let $G$ be a group and $N$ be a normal subgroup of $G$. $\calX%
	=(inc:N\hookrightarrow G)$ is a crossed module. Where $G$ acts on $N$ by
	conjugation.
	
	\item Let $M$ be a $G$-module. $\calX$ is a crossed module with trivial
	morphism $0_{M}:M\rightarrow G$, $m\mapsto 1_{G}$.
	
	\item For a group $H$, $\alpha :H\rightarrow Aut(H)$ homomorphism represents
	the action so that $\calX=(\alpha :H\rightarrow Aut(H))$ is a crossed module.
	
	\item A central extension crossed module has boundary as surjection $%
	\partial : S \to R$ with the central kernel, where $r \in R$ acts on $S$ by
	conjugation with $\partial^{-1}r$.
	
	\item The direct product of $\calX=(\partial :S\rightarrow R)$ and $\calX%
	^{\prime }=(\partial ^{\prime }:S^{\prime }\rightarrow R^{\prime })$ is $%
	\calX\times \calX^{\prime }=(\partial \times \partial ^{\prime }:S\times
	S^{\prime }\rightarrow R\times R^{\prime })$ with $R,R^{\prime }$ acting
	trivially on $S,\ S^{\prime }$ respectively.
\end{itemize}

Let $\calX =(\partial :S\rightarrow R)$ and $\calX^{\prime }=(\partial
^{\prime }:S^{\prime }\rightarrow R^{\prime })$ be two crossed modules. The
diagram
$$
%\xymatrixrowsep{0.7in}\xymatrixcolsep{0.7in}%
\xymatrix{
	& S \ar[d]_{\partial} \ar[r]^{\theta}
	& S^{\prime} \ar[d]^{\partial^{\prime}}  \\
	&R  \ar[r]_{\psi}
	&R^{\prime}            } 
$$

\noindent is commutative. That is for $r\in R$ and $s\in S$ 
\[
{\partial }^{\prime }(\theta (s))=\psi (\partial (s)) 
\]%
and $\theta $ preserves the action of $R$ on $S$.

\[
\theta (r\cdot s)=\psi (r)\cdot {\theta (s)} 
\]%
then $(\theta ,\psi )$ is called the morphism of crossed modules. This
states a category of crossed modules, \catXMod.

A simplicial group \textbf{G} consists of a family of groups $\{G_{n}\}$
together with face and degeneracy maps%
\begin{eqnarray*}
	d_{i} &=&d_{i}^{n}:G_{n}\rightarrow G_{n-1},\text{ \ \ }0\leq i\leq n,\text{ 
	}(n\neq 0) \\
	s_{i} &=&s_{i}^{n}:G_{n}\rightarrow G_{n+1},\text{ \ \ }0\leq i\leq n,\text{ 
	}(n\neq 0)
\end{eqnarray*}

satisfying the usual simplicial identities given in \cite{may,curtis}%
. The category of simplicial groups is denoted by \catSimpGrp.

The Moore complex $\mathbf{NG}$ of a simplicial group $\mathbf{G}$ is
defined to be the normal chain complex $(\mathbf{NG},\partial )$ with%
\[
NG_n= {\bigcap_{i=0}^{n-1} }\ker d_{i} 
\]%
and with $\partial _{n}:$ $NG_{n}\rightarrow $ $NG_{n-1}$ induced from $%
d_{n} $ by restriction. The n$^{th}$ \textit{homotopy group }$\pi _{n}$($%
\mathbf{G}$) of $\mathbf{G}$ is the n$^{th}$ homology of the Moore complex
of $\mathbf{G}$, i.e. 
\[
\pi _{n}(\mathbf{G})\cong H_{n}(\mathbf{NG},\partial )=\bigcap_{i=0}^{n}\ker
d_{i}^{n}/d_{n+1}^{n+1}(\bigcap_{i=0}^{n}\ker d_{i}^{n+1}). 
\]

We say that the Moore complex $\mathbf{NG}$ of a simplicial group $\mathbf{G}
$ is of \textit{length k} if $\mathbf{NG}_{n}=1$ for all $n\geq k+1$. We
denote the category of simplicial groups with Moore complex of length $k$ by 
$\mathbf{SimpGrp}_{\leq k}.$

\begin{theorem}\label{teo1} The category of crossed modules is equivalent to the category of simplicial groups with Moore complex of length 1. (see \cite{LO2})
\end{theorem}

\section{Group with Action}\label{sec-gwa}

In this section, we recall the definition of group with action given by
Datuashvili in \cite{DT}.

Let $G$ be a group. A map $\varepsilon :G\times G\rightarrow G$ represents
the right action on itself. For $g,g^{\prime },g^{\prime \prime }\in G$, 
\begin{eqnarray*}
	\varepsilon (g,g^{\prime }+g^{\prime \prime }) &=&\varepsilon (\varepsilon
	(g,g^{\prime }),g^{\prime \prime }) \\
	\varepsilon (g,0) &=&g \\
	\varepsilon (g^{\prime }+g^{\prime \prime },g) &=&\varepsilon (g^{\prime
	},g)+\varepsilon (g^{\prime \prime },g) \\
	\varepsilon (0,g) &=&0
\end{eqnarray*}%
\noindent with the above conditions, $G^{\bullet }$ is called the group with
action. The action is denoted by $\varepsilon (g,h)=g^{h},$ for $g,h\in G$.
Here the group operation is addition.

Let $(G,\varepsilon)$ and $(G^{\prime},\varepsilon^{\prime})$ be group with
actions. A morphism between group with actions is denoted with the following
diagram. 
\[
\xymatrixrowsep{0.7in}\xymatrixcolsep{0.7in}%
\xymatrix{ G \times G
	\ar[r]^{\varepsilon} \ar[d]_{(\varphi,\varphi)} & G \ar[d]^{\varphi} \\ G'
	\times G' \ar[r]^{\varepsilon'} & G' } 
\]
\noindent with $\phi:G \longrightarrow G^{\prime}$ map. Furthermore for $g,h
\in G$ 
\[
(G,\varepsilon_{G}) \longrightarrow (G^{\prime},\varepsilon_{G}^{\prime}) 
\]

\textbf{\ Examples}

\begin{itemize}
	\item Let $G$ be a group. Consider $G^{\bullet }$ as a group with action
	with the (right) action by conjugation.
	
	\item Consider the group klein four $Kl_{4}=\{e,a,b,ab\}$. We have ten
	groups with action obtained from $Kl_{4}$. Three of them are denoted by $%
	(Kl_{4},\varepsilon _{i})$, $i=1,2,3$. The tables of the actions $%
	\varepsilon _{i}$ are as follows: 
	\[
	\begin{tabular}{c||cccc}
		$\varepsilon _{1}$ & $e$ & $a$ & $b$ & $ab$ \\ \hline\hline
		$e$ & $e$ & $a$ & $b$ & $ab$ \\ 
		$a$ & $e$ & $a$ & $b$ & $ab$ \\ 
		$b$ & $e$ & $a$ & $ab$ & $b$ \\ 
		$ab$ & $e$ & $a$ & $ab$ & $b$%
	\end{tabular}%
	\ \qquad 
	\begin{tabular}{c||cccc}
		$\varepsilon _{2}$ & $e$ & $a$ & $b$ & $ab$ \\ \hline\hline
		$e$ & $e$ & $a$ & $b$ & $ab$ \\ 
		$a$ & $e$ & $ab$ & $b$ & $a$ \\ 
		$b$ & $e$ & $a$ & $b$ & $ab$ \\ 
		$ab$ & $e$ & $ab$ & $b$ & $a$%
	\end{tabular}%
	\qquad 
	\begin{tabular}{c||cccc}
		$\varepsilon _{3}$ & $e$ & $a$ & $b$ & $ab$ \\ \hline\hline
		$e$ & $e$ & $a$ & $b$ & $ab$ \\ 
		$a$ & $e$ & $a$ & $b$ & $ab$ \\ 
		$b$ & $e$ & $a$ & $b$ & $ab$ \\ 
		$ab$ & $e$ & $a$ & $b$ & $ab$%
	\end{tabular}%
	\]%
	where the $ij$th element shows the right action of the $i$th element on the $%
	j$th element for $i,j\in \{1,2,3,4\}$
\end{itemize}
\begin{definition}
	\label{ideal} Let $G^{\bullet }$ be a group with action and $A$ be a
	nonempty subset of $G$. If the conditions
	
	\begin{enumerate}
		\item[i)] $A$ is a normal subgroup of $G$ as a group,
		
		\item[ii)] $a^{g}\in A$, for $a\in A$ and $g\in G,$
		
		\item[iii)] $-g+g^{a}\in A$, for $a\in A$ and $g\in G$,
	\end{enumerate}
	
	\noindent satisfied, then $A^{\bullet }$ is called an ideal of $G^{\bullet }$ \cite{DT}.
\end{definition}

\textbf{Condition 1:}
For each $x,y,z\in G$, 
\[
x-x^{(z^{x})}+x^{y+z^{x}}-x+x^{z}-x^{z+y^{z}}=0\text{.} 
\]

In \cite{DT}, category of Abelian groups with action satisfying this
condition and category of Lie-Leibniz algebras were defined. Then it was
proved that the analogue of Witt's construction defined a functor from the
category of groups with action to category of Lie-Leibniz algebras which
gave rise to Leibniz algebras (introduced in \cite{LO} ) over the ring of
integers.

\begin{example}
	Each group with the trivial action satisfies Condition 1.
\end{example}

\section{Crossed Modules of Groups with Action}\label{sec-gwaXMod}

In this section, we will define an action between two group with actions and
a new category $\catXModwA$ called the category of crossed module of groups
with action. Let $S^{\bullet }=(S,\varepsilon _{S})$ and $R^{\bullet
}=(R,\varepsilon _{R})$ be groups with action. We can use the exact sequence
to define the action. For $s\in S,r\in R$ the \noindent action $(r,s)\mapsto
r\cdot s$ of $R$ on $S$ can be represented with the following sequence

\[
\xymatrixrowsep{0.3in}\xymatrixcolsep{0.4in}%
\xymatrix{\ & 0\ar[r]&
	S\ar[r]^{i} &K\ar[r]^{j} & R\ar@<0ex>[r]\ar@/^/[l]<1ex>^{{t}} & 0 } 
\]

Existence of the function $t$ with $jt=1_{R}$ is the main property for the
described exact sequence and the equation $r\cdot s=-t(r)+i(s)+t(r)$ denotes
the $R$-action on $S$. For $s,s_{1}\in S$ and $r,r_{1}\in R$

\begin{enumerate}
	\item[i)] $(r+r_{1})\cdot s=r\cdot (r_1\cdot s)$
	
	\item[ii)] $r\cdot (s+s_{1})=r\cdot s+r\cdot s_{1}$
	
	\item[iii)] $0_{R}\cdot s=s,$ $r\cdot 0_{S}=0_{s}$
\end{enumerate}

In this category, there must be two derived actions of $R$ on $S$
corresponding to the group operations. The First one is defined above. We define
the second action as $r \ast s = s^{t(r)}$.

The exact sequences with left group actions on itself can be represented by
the following diagram. 
\[
\xymatrixrowsep{0.3in}\xymatrixcolsep{0.3in}%
\xymatrix{ S \ar@{->}[d]_{i_{S}} \ar[rr]^{i} && K \ar@{=}[dd]^{id}
	\ar[rr]^{j} && R \ar@{->}[d]^{i_{R}} \ar @/_1pc/[ll]_{t} \\ B
	\ar@{->}[d]_{j_{S}} && && D \ar@{->}[d]_{j_{R}} \\ S \ar @/_1pc/[u]_{t_{S}}
	\ar[rr]_{i} && K\ar[rr]_{j} && R \ar @/_1pc/[u]_{t_{R}} \ar @/_1pc/[ll]_{t} }
\]%
We obtained the following equalities via the diagram. \label{cond_action}

\begin{enumerate}
	\item[iv)] $(r + r_1) \ast s = r \ast (r_1 \ast s)$
	
	\item[v)] $r \ast (s+s_1)= r \ast s+r \ast s_1$
	
	\item[vi)] $0_{R}\ast s=s,$ $r\ast 0_{S}=0_{s}$
	
	\item[vii)] $r\ast (r_{1}\cdot s)=(r_{1}^{r})\cdot (r\ast s)$
	
	\item[viii)] $(r\ast s)^{(r\cdot s_{1})}=r\ast s^{s_{1}}$
\end{enumerate}

\begin{definition}
	Let $S^{\bullet }=(S,\varepsilon _{S})$ and $R^{\bullet }=(R,\varepsilon
	_{R})$ be groups with action. Denote $\varepsilon _{S}(s,s_1)={s_1}^s$ and $%
	\varepsilon _{R}(r,r_1)={r_1}^r$ for $s,s_1\in S$ and $r,r_1\in R$. Let
	
	\[
	\begin{array}{ccc}
		\cdot :R\times S & \longrightarrow & S \\ 
		(r,s) & \mapsto & r\cdot s%
	\end{array}%
	\text{ \ \ \ \ and\ \ \ \ \ \ \ \ }%
	\begin{array}{ccc}
		\star :R\times S & \longrightarrow & S \\ 
		(r,s) & \mapsto & r\star s%
	\end{array}%
	\]%
	be actions of $R$ on $S$. We denote the group operation additively,
	nevertheless the group is not abelian. Moreover, let $\varphi _{S}$ and $%
	\varphi _{R}$ be conjugate actions on $S$ and $R$ respectively. Using
	commutativity of the diagrams,
	
	\[
	\begin{array}[t]{cc}
		\begin{array}{c}
			\xymatrixrowsep{4em}\xymatrixcolsep{4em}\xymatrix{ S\times S
				\ar[r]^-{\varphi_S} \ar[d]_{\partial \times 1_S}& S \ar[d]^{1_S} \\ R\times
				S \ar[r]^-{\cdot} \ar[d]_{ 1_R \times \partial}& S \ar[d]^{\partial } \\
				R\times R \ar[r]^-{\varphi_R} & R}%
		\end{array}
		& 
		\begin{array}{c}
			\xymatrixrowsep{4em}\xymatrixcolsep{4em}\xymatrix{ S\times S
				\ar[r]^-{\varepsilon_S} \ar[d]_{\partial \times 1_S}& S \ar[d]^{1_S} \\
				R\times S \ar[r]^-{\star} \ar[d]_{ 1_R \times \partial}& S \ar[d]^{\partial
				} \\ R\times R \ar[r]^-{\varepsilon_R} & R}%
		\end{array}%
	\end{array}%
	\]
	
	$\partial: (S,\varepsilon _{S}) \longrightarrow (R,\varepsilon _{R})$ is a crossed module in the category of groups with
	action by following conditions
	
	\[
	\begin{tabular}{llll}
		\textbf{CM1)}$\ $ & $\partial (r\cdot s)$ & $=$ & $\varphi _{R}(r,\partial
		(s))$ \\ 
		&  & $=$ & $-r+\partial (s)+r$ \\ 
		\textbf{CM2)}$\ $ & $\partial (s)\cdot s_{1}$ & $=$ & $1_{S}(\varphi
		_{S}(s,s_{1}))$ \\ 
		&  & $=$ & $\varphi _{S}(s,s_{1})$ \\ 
		&  & $=$ & $-s+s_{1}+s$ \\ 
		\textbf{CM3)}$\ \ $ & $\partial (r\star s)$ & $=$ & $\varepsilon
		_{R}(r,\partial (s))$ \\ 
		&  & $=$ & ${\partial (s)}^{r}$ \\ 
		\textbf{CM4)}$\ $\  & $\partial (s)\star s_{1}$ & $=$ & $1_{S}(\varepsilon
		_{S}(s,s_{1}))$ \\ 
		&  & $=$ & $\varepsilon _{S}(s,s_{1})$ \\ 
		&  & $=$ & ${s_{1}}^{s}$%
	\end{tabular}%
	\]
\end{definition}

A crossed module of groups with action is written by $\calX^{\bullet
}=(\partial :S^{\bullet }\rightarrow R^{\bullet })$ notation. When only 
\textbf{CM1 }and \textbf{CM3 }conditions are satisfied, the resulting
structure is a \emph{pre-crossed module}.

\noindent \textbf{Examples}

\begin{itemize}
	\item Let $R^{\bullet }$ be a group with action and $S^{\bullet }$ be an
	ideal of $R^{\bullet }$. $\calX^{\bullet }=(inc:S^{\bullet }\hookrightarrow
	R^{\bullet })$ is a crossed module of groups with action. Where $R$ acts on $%
	S$ by $r\cdot s=-r+s+r$ and $r\star s=\varepsilon _{R}(r,s)=s^{r}$.
	
	\item Let $G^{\bullet }$ be a group with action and $M$ be any $G$-module.
	Using $\varepsilon _{M}(m,m_{1})=m_{1}^{m}=m_{1}$ trivial action $M^{\bullet
	}$ is a group with action. Then $\calX^{\bullet }=(0:M^{\bullet }\rightarrow
	G^{\bullet })$ is a crossed module of groups with action. Where the boundary
	of $\calX^{\bullet }$ is $(m,g)\mapsto e_{G}$ zero morphism and $G$ acts on $%
	M$ by $(g,m)\rightarrow g\cdot m$ and $(g,m)\rightarrow g\star m$ any two
	actions.
	
	\item The direct product of $\calX_{1}^{\bullet }=(\partial
	_{1}:S_{1}^{\bullet }\rightarrow R_{1}^{\bullet })$ and $\calX_{2}^{\bullet
	}=(\partial _{2}:S_{2}^{\bullet }\rightarrow R_{2}^{\bullet })$ is $\calX%
	_{1}^{\bullet }\times \calX_{2}^{\bullet }=(\partial _{1}\times \partial
	_{2}:S_{1}^{\bullet }\times S_{2}^{\bullet }\rightarrow R_{1}^{\bullet
	}\times R_{2}^{\bullet })$ a crossed module of groups with action with $%
	R_{1},R_{2}$ acting trivially on $S_{1},\ S_{2}$ respectively.
\end{itemize}

\begin{definition}
	Let $\calX_{1}^{\bullet }=(\partial _{1}:S_{1}^{\bullet }\rightarrow
	R_{1}^{\bullet })$ and $\calX_{2}^{\bullet }=(\partial _{2}:S_{2}^{\bullet
	}\rightarrow R_{2}^{\bullet })$ be crossed modules of groups with action. A
	crossed module of groups with action morphism%
	\[
	(\alpha ,\beta ):(\partial _{1}:S_{1}^{\bullet }\rightarrow R_{1}^{\bullet
	})\longrightarrow (\partial _{2}:S_{2}^{\bullet }\rightarrow R_{2}^{\bullet
	}) 
	\]
\end{definition}

\noindent is a pair of homomorphisms $\alpha :S_{1}^{\bullet }\rightarrow
S_{2}^{\bullet }$, $\beta :R_{1}^{\bullet }\rightarrow R_{2}^{\bullet }$
such that%
\[
\begin{tabular}{llll}
	i)$\ $ & $\beta \partial _{1}(s_{1})=\partial _{2}\alpha (s_{1})$ &  & for
	all $s_{1}\in S_{1}$ \\ 
	ii)$\ $ & $\alpha (r_{1}\cdot _{1}s_{1})=\beta (r_{1})\cdot _{2}\alpha
	(s_{1})$ &  & for all $s_{1}\in S_{1},r_{1}\in R_{1}$ \\ 
	iii)$\ \ $ & $\alpha (r_{1}\star _{1}s_{1})=\beta (r_{1})\star _{2}\alpha
	(s_{1})$ &  & 
\end{tabular}%
\]%
So we get the category of crossed module of groups with action. It is
denoted by $\catXModwA$. If $(\alpha ,\beta ):(\partial _{1}:S_{1}^{\bullet
}\rightarrow R_{1}^{\bullet })\longrightarrow (\partial _{2}:S_{2}^{\bullet
}\rightarrow R_{2}^{\bullet })$ is a crossed module of groups with action
morphism such that $\alpha $ and $\beta $ both isomorphisms then $(\alpha
,\beta )$ is called an isomorphism. The kernel of $(\alpha ,\beta )$, $\ker
(\alpha ,\beta )$ is the crossed module of groups with action $(\partial
:\ker \alpha \rightarrow \ker \beta )$.

\section{Simplicial Group with Action}

\begin{definition}
	Let $[n]$ be an ordered set. For $[n]$ and $[m]$ ordered sets, $f:[n]
	\longrightarrow [m]$ monotone function is called the operator. $\Delta[n]$
	category have $[n]$ objects and $f$ operator morphism. $\Delta^{op}[n]$ is
	the dual of the $\Delta[n]$ category.
	
	Let $\catGrwA$ be the category of group with action. The functor
	
	\[
	G^{\bullet }:\Delta ^{op}[n]\longrightarrow \catGrwA 
	\]%
	is the simplicial group with action.
	
	The operator of the simplicial category can be represented with two special
	operators, $\delta _{i}^{n}$ and $\sigma _{j}^{n}$: 
	\[
	G^{\bullet }([n])=G_{n}^{\bullet },G^{\bullet }(\delta _{i}^{n})=d_{i}^{n}%
	\text{ and }G^{\bullet }(\sigma _{j}^{n})=s_{j}^{n} 
	\]%
	The category of simplicial group with action can be defined and is denoted
	by $\catSimpGrpwA$.
\end{definition}

\begin{definition}
	Let $G^{\bullet }$ be a simplicial group with action. Let
	
	\[
	NG_{n}^{\bullet }=\bigcap_{i=0}^{n-1}\ker d_{i}^{n} 
	\]%
	and the restriction of function $d_{n}^{n}$ 
	\[
	\partial _{n}:NG_{n}^{\bullet }\longrightarrow NG_{n-1}^{\bullet } 
	\]%
	is defined. So the chain complex 
	$$
	\xymatrix{
		(NG^{\bullet}):...NG_2^{\bullet} \ar[r]^-{\partial_2 }& NG_1^{\bullet} \ar[r]^-{\partial_1}& NG_0^{\bullet}=G_0^{\bullet}
	}
	$$
	is called the Moore chain complex of simplicial group with action. In the
	chain complex 
	\[
	NG_{0}^{\bullet }=G_{0}^{\bullet },NG_{1}^{\bullet }=\ker d_{0}^{\prime
	},NG_{2}^{\bullet }=\ker d_{0}^{2}\cap \ker d_{1}^{2}, 
	\]%
	are denoted.
\end{definition}

The following result of different versions (such as for Lie, group, algebra)
can be found in \cite{arvasi-porter1,TIM}

\begin{theorem}
	The category of crossed module of groups with action $\catXModwA$ is
	equivalent to the category of simplicial group with action $\catSimpGrpwA$
	with Moore Complex of length 1.
\end{theorem}

\begin{proof} Let $G^{\bullet }$ be a simplicial group with
	action with Moore complex of length 1. For the group with action crossed
	modules, 
	\[
	NG_{2}^{\bullet }\cap D_{2}=0\text{ and }\partial _{2}(NG_{2}^{\bullet }\cap
	D_{2})=\partial _{2}(G_{2}^{\bullet }\cap D_{2})=\{0\} 
	\]%
	is satisfied. $D_{2}$ is a subgroup generated by $s_{j}$ degenerated
	operator in $G_{2}^{\bullet }$. 
	$$
	\xymatrix{
		(NG^{\bullet}):...NG_2^{\bullet} \ar[r]^-{\partial_2 }& NG_1^{\bullet} \ar[r]^-{\partial_1}& NG_0^{\bullet}
	}
	$$
	chain complex has a sub-complex as 
	\[
	NG_{1}^{\bullet }\longrightarrow NG_{0}^{\bullet } 
	\]%
	defined sub-complex is a group with action homomorphism. An action $%
	NG_{0}^{\bullet }$ on $NG_{1}^{\bullet }$ 
	\[
	\begin{array}{ccc}
		NG_{0}^{\bullet }\times NG_{1}^{\bullet } & \longrightarrow & 
		NG_{1}^{\bullet } \\ 
		(r,s) & \longrightarrow & s^{r}=-(s_{0}r)+s+(s_{0}r) \\ 
		&  &  \\ 
		NG_{0}^{\bullet }\times NG_{1}^{\bullet } & \longrightarrow & 
		NG_{1}^{\bullet } \\ 
		(r,s) & \longmapsto & r \ast s=s^{s_{0}(r)}%
	\end{array}%
	\]%
	\noindent is constructed. For the first condition of group with action of
	crossed modules, 
	\begin{eqnarray*}
		\partial _{1}(s^{r}) &=&\partial _{1}(-(s_{0}r)+s+(s_{0}r))=-r+d_{1}s+r
		\\
		&=&-r+\partial _{1}s+r
	\end{eqnarray*}%
	\noindent is satisfied properly. On the other hand, let $a,b\in
	NG_{1}^{\bullet }$. 
	\[
	b^{\partial _{1}a}=-s_{0}d_{1}a+b+s_{0}d_{1}a 
	\]%
	is defined by the action of group with action. $NG_{2}^{\bullet }\cap
	D_{2}^{\bullet }=N_{2}^{\bullet }\cap D_{2}^{\bullet }$ equality is
	satisfied with $a,b\in NG_{1}^{\bullet }$ and 
	\[
	F_{(0)(1)}(a,b)=[s_{0}a,s_{1}b][s_{1}b,s_{1}a]\in NG_{2}^{\bullet }\cap
	D_{2}^{\bullet } 
	\]%
	in $NG_{2}^{\bullet }\cap D_{2}^{\bullet }$. Since the length of Moore
	complex is 1, we have 
	\[
	\partial _{2}(N_{2}^{\bullet }\cap D_{2}^{\bullet })=\{0\} 
	\]%
	and the following equality 
	\begin{eqnarray*}
		\partial _{2}(F_{(0)(1)}(a,b)) &=&d_{2}([s_{0}a,s_{1}b][s_{1}b,s_{1}a]) \\
		&=&-s_{0}d_{1}a-b+s_{0}d_{1}a+b(-b-a+b+a)\in NG_{2}^{\bullet
		}\cap D_{2}^{\bullet }
	\end{eqnarray*}%
	so 
	\[
	-s_{0}d_{1}a+b+s_{0}d_{1}a=-a+b+a 
	\]%
	\noindent is the second condition of group with action crossed modules. For $%
	r \in NG_{0}^{\bullet},s \in NG_{1}^{\bullet}$ 
	\begin{eqnarray*}
		\partial _{1}(s^{s_{0}(\partial_{1}r)}) &=&\partial _{1}(s)^{\partial
			_{1}(r)} \\
		&=&\partial _{1}(s^r)
	\end{eqnarray*}
	
	\noindent Since $\partial _{1}$ is one-to-one morphism, crossed module of
	group with action CM4 condition holds.
	
	CM3)%
	\begin{eqnarray*}
		\partial _{1}(r \ast s) &=&\partial _{1}(s^{s_{0}(r)}) \\
		&=&\partial _{1}(s)^{\partial _{1}(s_{0}r)} \\
		&=& \partial _{1}(s)^r \\
		&=& \partial _{1}(s)\ast r
	\end{eqnarray*}
	
	\noindent condition is obtained.
	
	\noindent Thus, there exists an obvious functor 
	\[
	\begin{array}{llll}
		\triangle : & \catXModwA & \rightarrow & \catSimpGrpwA_{\leq 1} \\ 
		& \qquad R^{\bullet } & \mapsto & \triangle (R^{\bullet })=(S^{\bullet
		}\rightarrow R^{\bullet })%
	\end{array}%
	\]%
	can simply be noticed. Conversely, let $S^{\bullet }$ and $R^{\bullet }$ be
	a group with actions. We will show that 
	\[
	\partial :S^{\bullet }\longrightarrow R^{\bullet } 
	\]%
	is a crossed module of group with action. Since $R^{\bullet }$ acts on $%
	S^{\bullet }$ we far from the semi-direct product $R^{\bullet }\rtimes
	S^{\bullet }$. We define the semi-direct group with action $R^{\bullet
	}\rtimes S^{\bullet }$ by 
	\[
	(r,s)(r_{1},s_{1})=(r+r_{1 },s^ {r_{1}}+s_{1 }) 
	\]%
	and 
	\[
	(r,s)^{(r_{1},s_{1})}=(r \ast{r_{1}},r \ast s_{1}+s \ast s_{1}+s \ast r_{1}) 
	\]
	an action on $S^{\bullet}$ by $R^{\bullet}$. Defining $S_{0}^{\bullet
	}=R^{\bullet }$ and $S_{1}^{\bullet }=R^{\bullet }\rtimes S^{\bullet }$ we
	get 
	\begin{eqnarray*}
		d_{1}(r,s) &=&(\partial _{1}s)+r \\
		d_{0}(r,s) &=&r \\
		s_{0}(r) &=&(0,r)
	\end{eqnarray*}%
	\noindent These operators satisfy the following simplicial identities due to being group with action morphisms.  
	\begin{eqnarray*}
		d_{1}s_{0}(r) &=&d_{1}(r,0)=(\partial _{1}0)+r=r \\
		d_{0}s_{0}(r) &=&d_{0}(r,0)=r
	\end{eqnarray*}%
	Thus, a 1-truncated simplicial group with action is $\{S_{1}^{\bullet
	},S_{0}^{\bullet }\}=S^{\prime \bullet }$. From truncation we have the
	following equalities 
	\begin{eqnarray*}
		NS_{0}^{\prime \bullet } &=&NS_{0}^{\bullet }=S_{0}^{\bullet }=N \\
		NS_{1}^{\prime \bullet } &=&NS_{1}^{\bullet }=(\ker d_{0})NS_{2}^{\prime
			\bullet }=\{0\}
	\end{eqnarray*}%
	Equivalently, if 
	\[
	\partial _{2}(NS_{2}^{\bullet }\cap D_{2}^{\bullet })=\partial
	_{2}(NS_{2}^{\bullet })=\{0\} 
	\]%
	is satisfied, $S^{\prime \bullet }$ is denoted a simplicial group with
	action with Moore complex of length 1. By the definition of $F_{\alpha
		,\beta }$ functions 
	\[
	\partial _{2}(NS_{2}^{\bullet })=[\ker d_{0},\ker d_{1}] 
	\]%
	is obtained. Also 
	\[
	\begin{tabular}{lll}
		$d_{0}:R^{\bullet }\rtimes S^{\bullet }$ & $\longrightarrow $ & $R^{\bullet
		} $ \\ 
		\multicolumn{1}{r}{$(r,s)$} & $\mapsto $ & $r$%
	\end{tabular}%
	\]%
	with $\ker d_{0}=\{(0,s):s\in S^{\bullet }\}\cong S^{\bullet }$ and 
	\[
	\begin{tabular}{lll}
		$d_{1}:R^{\bullet }\rtimes S^{\bullet }$ & $\longrightarrow $ & $R^{\bullet
		} $ \\ 
		\multicolumn{1}{r}{$(r,s)$} & $\mapsto $ & $(\partial _{1}s)+r$%
	\end{tabular}%
	\]%
	with $\ker d_{1}=\{(-\partial _{1}s,s):s\in S^{\bullet }\}$ are
	conveniently determined. Since $(0,s_{1})\in kerd_{0}$ and \newline
	$(-\partial _{1}s,s)\in \ker d_{1}$, we have 
	\[
	\lbrack (0,s_{1}),(-\partial _{1}s,s)]=(0,0). 
	\]%
	Using the definition of semi-direct product the following equation holds 
	\begin{eqnarray*}
		\lbrack (0,s_{1}),(-\partial _{1}s,s)] &=&(0,s_{1})(-\partial
		_{1}s,s)(-(0,s_{1}))(-(-\partial _{1}s,s)) \\
		&=&(-\partial_{1}s,(s_{1})^{-\partial_{1}(s)}+s)(0,-s_{1})(%
		\partial_{1}s,(-s)^{\partial _{1}s}) \\
		&=&(-\partial _{1}s,(s_{1})^{-\partial _{1}s}+s)(\partial
		_{1}s,(-s_{1})^{\partial _{1}s}+(-s)^{\partial _{1}s}) \\
		&=&(0,((s_{1})^{-\partial_{1}s}+s)^{\partial_{1}s}+(-s_{1})^{%
			\partial _{1}s}+(-s)^{\partial _{1}s}) \\
		&=&(0_{R},0_{S})
	\end{eqnarray*}%
	where $(0_{R^{\bullet }},0_{S^{\bullet }})$ is the identity element.
	
	Furthermore, 
	\begin{eqnarray*}
		(0,s_{1})^{(-\partial_{1}s,s)}&=&(0 \ast{(-\partial_{1}s)},0 \ast s+
		s_{1} \ast s+ s_{1}*{(-\partial_{1}s)}) \\
		&=& (0,\partial_{1}s_{1} \ast s+\partial_{1}s_{1} \ast {(-s)}) \\
		&=& (0, s^{s_{1}}+(-s)^{s_{1}}) \\
		&=& (0,s_{1})
	\end{eqnarray*}
	and 
	\begin{eqnarray*}
		(-\partial_{1}s,s)^{(0,s_{1})}&=& (-\partial_{1}s,-\partial_{1}s
		\ast s_{1} + s \ast s_{1} +s \ast 0) \\
		&=& (-\partial_{1}s,(-s) \ast s_{1} + s \ast s_{1} +s) \\
		&=& (-\partial_{1}s,0 \ast s_{1} + s) \\
		&=& (-\partial_{1}s, s)
	\end{eqnarray*}
	equations are hold. So, 
	\[
	\lbrack \ker d_{0},\ker d_{1}]=(0_{R^{\bullet }},0_{S^{\bullet }}) 
	\]%
	and 
	\[
	\partial _{2}(NS_{2}^{\bullet })=\{0\} 
	\]%
	are obtained from $S^{\prime }=\{S_{1}^{\bullet },S_{0}^{\bullet }\}$ with
	Moore complex of length 1.
\end{proof}

\section{Computer Implementation}\label{sec-computer}

\GAP\ (Groups, Algorithms, Programming \cite{gap}) is the leading symbolic
computation system for solving computational discrete algebra problems.
Symbolic computation has underpinned several key advances in Mathematics and
Computer Science, for example, in number theory and coding theory (see \cite%
{hpc-gap} ). \GAP\, which is free, opensource, and extensible system, can
deal with different discrete mathematical problems, but it focuses on
computational group theory. It is distributed under the GNU Public License.
The system is delivered together with the source codes, which are written in
two languages: the kernel of the system is written in C, and the library of
functions and additional packages is in a special language, also called \GAP%
. The \GAP\ system and extension packages now comprise 360K lines of C and
900K lines of \GAP\ code. The Small Groups library has been used in such
landmark computations as the "Millennium Project" to classify all finite
groups of order up to 2000 by Besche, Eick and O'Brien in \cite{besche}.

The \textsf{GwA} package for GAP contains functions for groups with action and
their morphisms and was first described in \cite{gwa}. In this paper, we
have developed new functions which construct (pre) crossed modules of groups
with action and renamed the package to \textsf{XModGwA}.

The function \textbf{GwA} in the package may be used in two ways. \textbf{%
	GwA($G$)} returns the group with trivial action, while \textbf{GwA($G$,$act$)%
} returns a group with action for this chosen group and an $act
:G\rightarrow Aut(G)$ action. Functions for groups with action include 
\textbf{IsGwA}, \textbf{IsPerfectGwA}, \textbf{IsIdeal}, \textbf{%
	AllIdealOnGwA}, \textbf{IsGwAC1}, \textbf{Commutator}, \textbf{%
	LowerCentralSeriesOfGwA}, \textbf{IsNilpotent}, and \textbf{%
	NilpotencyClassOfGwA}. Attributes of a group with action constructed in this
way include \textbf{BaseGroup} and \textbf{BaseAction}.

The function \textbf{GwAMorphismObj} defines morphisms of groups with
action, and the function \textbf{IsGwAMorphism} which controls whether a map
satisfies condition morphisms of groups with action or not. The function 
\textbf{AllGwAMorphisms} is used to find all morphisms between two groups
with action.

The following GAP session illustrates the use of these functions.

\begin{Verbatim}[frame=single, fontsize=\small, commandchars=\\\{\}]
	\textcolor{blue}{gap> kl4 := Group((1,2),(3,4));;}
	\textcolor{blue}{gap> allgwa_on_kl4 := AllGwAOnGroup(kl4);;}
	\textcolor{blue}{gap> Length(allgwa_on_kl4);}
	10
	\textcolor{blue}{gap> kl4_A := allgwa_on_kl4[4];}
	GroupWithAction [ Group( [ (1,2), (3,4) ] ), * ]
	\textcolor{blue}{gap> A4 := AlternatingGroup(4);;}
	\textcolor{blue}{gap> A4_A := GwA(A4);}
	GroupWithAction [ AlternatingGroup( [ 1 .. 4 ] ), * ]
	\textcolor{blue}{gap> IsGwA(kl4_A); IsGwA(A4_A);}
	true
	true
	\textcolor{blue}{gap> IsPerfectGwA(kl4_A);}
	false
	\textcolor{blue}{gap> List(allgwa_on_kl4, i -> NilpotencyClassOfGwA(i));}
	[ 1, 2, 0, 0, 0, 2, 0, 0, 0, 2 ]
	\textcolor{blue}{gap> LowerCentralSeriesOfGwA(A4_A);}
	[ GroupWithAction [ AlternatingGroup( [ 1 .. 4 ] ), * ],
	GroupWithAction [ Group( [ (), (1,2)(3,4), (1,3)(2,4), (1,4)(2,3) ] ), * ] ]
	\textcolor{blue}{gap> gen_kl4 := GeneratorsOfGroup(kl4);;}
	\textcolor{blue}{gap> gen_A4 := GeneratorsOfGroup(A4);;}
	\textcolor{blue}{gap> el_A4 := Elements(A4);;}
	\textcolor{blue}{gap> f := GroupHomomorphismByImages(kl4,A4,gen_kl4,[el_A4[4],el_A4[4]]);}
	[ (1,2), (3,4) ] -> [ (1,2)(3,4), (1,2)(3,4) ]
	\textcolor{blue}{gap> m := GwAMorphismObj(kl4_A,A4_A,f);}
	GroupWithAction [ Group(
	[ (1,2), (3,4) ] ), * ] => GroupWithAction [ AlternatingGroup( [ 1 .. 4 ] ), * ]
	\textcolor{blue}{gap> IsGwAMorphism(m);}
	true
\end{Verbatim}

The function \textbf{AllGwAOnGroup($G$)} constructs a list of all groups
with action over $G$. The function \textbf{AreIsomorphicGwA} is used for
checking whether or not two groups with action are isomorphic, and \textbf{%
	IsomorphicGwAFamily} returns a list of representatives of the isomorphism
classes.

In the following \GAP\ session, we compute all $736$ groups with action on $%
C_{2}\times C_{2}\times C_{2}$; representatives of the $14$ isomorphism
classes; and the list of members of a family.

\begin{Verbatim}[frame=single, fontsize=\small, commandchars=\\\{\}]
	\textcolor{blue}{gap> C := Range(IsomorphismPermGroup(SmallGroup(8,5)));;}
	\textcolor{blue}{gap> allgwa_onC := AllGwAOnGroup(C);; 	}	
	\textcolor{blue}{gap> Length(allgwa_onC);}
	736
	\textcolor{blue}{gap> AreIsomorphicGwA(allgwa_onC[1],allgwa_onC[2]);}
	false
	\textcolor{blue}{gap> IsomorphicGwAFamily(allgwa_onC[101],allgwa_onC);}
	[ 101, 217, 323, 403, 490, 561, 576 ] 
\end{Verbatim}

Six of the isomorphism families satisfy Condition 1 and there are six
families with nilpotent and eight are not nilpotent. Other features obtained
with these functions are given in the table below.

\begin{longtable}{ccccccc}
	\hline\hline
	Family & Number of Members & Representator & Number of Ideals & Nilpotency Class & Condition 1    \\ \hline
	1 & 1 & 1/736 & 16 & 1 & \cmark    \\
	2 & 84 & 2/736 & 7 & 0 & \xmark    \\
	3 & 21 & 4/736 & 8 & 2 & \cmark    \\
	4 & 42 & 9/736 & 6 & 2 & \cmark    \\
	5 & 84 & 65/736 & 6 & 0 & \xmark    \\
	6 & 168 & 67/736 & 6 & 0 & \xmark    \\
	7 & 14 & 71/736 & 6 & 2 & \cmark    \\
	8 & 21 & 81/736 & 6 & 2 & \cmark    \\
	9 & 7 & 101/736 & 6 & 2 & \cmark    \\
	10 & 56 & 107/736 & 3 & 0 & \xmark    \\
	11 & 84 & 108/736 & 5 & 0 & \xmark    \\
	12 & 42 & 110/736 & 4 & 0 & \xmark    \\
	13 & 84 & 112/736 & 4 & 0 & \xmark    \\
	14 & 28 & 122/736 & 6 & 0 & \xmark   \\ \hline
	&  &  &  &  &   \\
\end{longtable}

Function for the action between two groups with action include \textbf{IsGwAAction}.
The function \textbf{IsGwAAction} is implemented for checking the structure of a group with action. 

The group $R$ acts on $S$ by an action $\alpha ,\beta :R\rightarrow Aut(S)$
and let $S^{\bullet }$ and $R^{\bullet }$ be groups with action on group $S$
and $R$, respectively. \textbf{IsGwAAction(}$S^{\bullet },R^{\bullet
},\alpha ,\beta $\textbf{)} is used to verify that the contidions of action
between two groups with action in page \pageref{cond_action} are satisfied.
The function \textbf{AllXModGwAActions(}$S^{\bullet },R^{\bullet }$\textbf{)}
constructs a list of all actions of group with action $R^{\bullet }$ on
group with action $S^{\bullet }$.

The following GAP session illustrates the use of these functions.

\begin{Verbatim}[frame=single, fontsize=\small, commandchars=\\\{\}]
	\textcolor{blue}{gap> G := Range(IsomorphismPermGroup(SmallGroup(8,2)));;}
	\textcolor{blue}{gap> allgwa_onG := AllGwAOnGroup(G);;}
	\textcolor{blue}{gap> Length(allgwa_onG);}
	32
	\textcolor{blue}{gap> SwA := allgwa_onG[2];}
	GroupWithAction [ Group( [ (1,2), (3,4), (5,6) ] ), * ]
	\textcolor{blue}{gap> RwA := allgwa_onG[4];}
	GroupWithAction [ Group( [ (1,2), (3,4,5,6) ] ), * ]
	\textcolor{blue}{gap> all_acts := AllXModGwAActions(SwA,RwA);;}
	\textcolor{blue}{gap> Length(all_acts);}
	256
	\textcolor{blue}{gap> act_pair := all_acts[13];;}
	\textcolor{blue}{gap> IsGwAAction(SwA,RwA,act_pair[1],act_pair[2]);}
	true
\end{Verbatim}

Functions for crossed modules of groups with action include \textbf{%
	PreXModGwAObj}, \textbf{IsPreXModGwA}, \textbf{IsXModGwA} and \textbf{%
	XModGwAByIdeal}. Attributes of a group with action constructed in this way
include \textbf{XModGwAAction}, \textbf{Range}, \textbf{Source} and \textbf{%
	Boundary}.

A structure which has \textbf{IsXModGwAC1} is a pre-crossed module or a
crossed module of groups with action whose source and range are both
satisfies condition 1.

The following GAP session illustrates the use of these functions.

\begin{Verbatim}[frame=single, fontsize=\small, commandchars=\\\{\}]
	\textcolor{blue}{gap> all_bdys := AllGwAMorphisms(SwA,RwA);;}
	\textcolor{blue}{gap> bdy := all_bdys[7];}
	GroupWithAction [ Group( [ (1,2), (3,4,5,6) ] ), * ] => 
	GroupWithAction [ Group( [ (1,2), (3,4,5,6) ] ), * ]
	\textcolor{blue}{gap> XM1 := PreXModGwAObj(bdy,act_pair[1], act_pair[2]);}
	GroupWithAction [ Group( [ (1,2), (3,4,5,6) ] ), * ] => 
	GroupWithAction [ Group( [ (1,2), (3,4,5,6) ] ), * ]
	\textcolor{blue}{gap> IsPreXModGwA(XM1);}
	true
	\textcolor{blue}{gap> IsXModGwA(XM1);}
	Condition 4 is fail
	For s = (1,2) and s1 = (1,2) => (1,2) <> (1,2)(3,5)(4,6)
	false
	\textcolor{blue}{gap> all_ideals := AllIdealOnGwA(RwA);;}
	\textcolor{blue}{gap> Length(all_ideals);}
	6
	\textcolor{blue}{gap> SwA := all_ideals[3];}
	GroupWithAction [ Group( [ (1,2)(3,4,5,6), (3,5)(4,6) ] ), * ]
	\textcolor{blue}{gap> XM2 := XModGwAByIdeal(RwA,SwA);}
	GroupWithAction [ Group( [ (1,2)(3,4,5,6), (3,5)(4,6) ] ), * ] => 
	GroupWithAction [ Group([ (1,2), (3,4,5,6) ] ), * ]
	\textcolor{blue}{gap> IsXModGwA(XM2);}
	true
	\textcolor{blue}{gap> IsXModGwAC1(XM2);}
	true
\end{Verbatim}

The global function \textbf{AllXModsGwA} may be called in two ways: as 
\textbf{AllXModsGwA(}$S^{\bullet },R^{\bullet }$\textbf{)} to compute all
crossed modules with chosen source and range groups with action; as \textbf{%
	AllXModsGwA(}$x,y,m,n$\textbf{)} to compute all crossed modules of groups
with actions with given size and numbers of small groups. The function
computes both all pre-crossed modules and crossed modules of groups with
action.

In the following GAP session, we get crossed modules using the function.

\begin{Verbatim}[frame=single, fontsize=\small, commandchars=\\\{\}]
	\textcolor{blue}{gap> all_XM3 := AllXModsGwA(SwA,RwA);;}
	\textcolor{blue}{gap> Length(all_XM3[1]); Length(all_XM3[2]);}
	66
	10
	\textcolor{blue}{all_XM4 := AllXModsGwA(4,1,4,2);;}
	\textcolor{blue}{Length(all_XM4[1]); Length(all_XM4[2]);}
	416
	184
	\textcolor{blue}{gap> list := Filtered(all_XM4[2], XM -> IsXModGwAC1(XM));;}
	\textcolor{blue}{gap> Length(list);}
	88
\end{Verbatim}